\documentclass[a4paper,11pt]{amsart}

\usepackage{amsmath,amssymb,amsfonts,amsthm}
\usepackage{eqnarray}
\usepackage{tikz, tikz-cd}
\usepackage{hyperref}
\usepackage{caption, subcaption} 
\usepackage{enumerate}

\usepackage{thmtools}
\usepackage{thm-restate}

\newcommand{\vs}{\vspace{0.4cm}}

\newcommand{\Z}{\mathbb{Z}}

\newtheorem{thm}{Theorem}[section]
\newtheorem{lem}[thm]{Lemma}
\newtheorem{prop}[thm]{Proposition}
\newtheorem{cor}[thm]{Corollary}

\theoremstyle{definition}
\newtheorem{df}[thm]{Definition}

\theoremstyle{remark}

\numberwithin{equation}{section}

\begin{document}

\title{Manifolds with odd Euler characteristic and higher orientability}

\author[R. S. Hoekzema]{Renee S. Hoekzema}
\address{Department of Mathematics, University of Oxford}

\email{hoekzema@maths.ox.ac.uk}
\thanks{This research was supported by the Hendrik Mullerfonds and the Fundatie van Renswoude}


\date{}

\dedicatory{}

\begin{abstract}
	It is well-known that odd-dimensional manifolds have Euler characteristic zero. Furthemore orientable manifolds have an even Euler characteristic unless the dimension is a multiple of $4$. We prove here a generalisation of these statements: a $k$-orientable manifold (or more generally Poincar\'e complex) has even Euler characteristic unless the dimension is a multiple of $2^{k+1}$, where we call a manifold $k$-orientable if the $i^{th}$ Stiefel-Whitney class vanishes for all $0<i< 2^k$ ($k\geq 0$). More generally, we show that for a $k$-orientable manifold the Wu classes $v_l$ vanish for all $l$ that are not a multiple of $2^k$.
	For $k=0,1,2,3$, $k$-orientable manifolds with odd Euler characteristic exist in all dimensions $2^{k+1}m$, but whether there exist a 4-orientable manifold with an odd Euler characteristic is an open question.
\end{abstract}

\maketitle

\section{Introduction}

Manifolds are central objects of study in geometry and topology. 
A basic property of a manifold is whether it is \emph{orientable}.  If an $n$-dimensional manifold $M$ is closed and connected, then $M$ is orientable if and only if its $n^{th}$ homology group $H_n(M;\Z)$ is isomorphic to $\Z$. An orientation on $M$ is a choice of a generator for this group.
A manifold is also orientable if and only if the first Stiefel-Whitney class $w_1$ of the tangent bundle of $M$ vanishes. This characteristic class is the obstruction to lifting the classifying map $\tau: M \rightarrow BO(n)$ of the tangent bundle to $BSO(n)$, which is the double cover of $BO(n)$.

A more restrictive property of a manifold, extending orientability, is \emph{spinability}.
We call a manifold \emph{spinable} if $\tau$ admits a lift to $BSpin(n)$, the 3-connected cover of $BO(n)$. For an orientable manifold the obstruction for this lift is given by the second Stiefel-Whitney class $w_2$. Hence a manifold is spinable precisely if $w_1=w_2=0$. 

Orientability and spinability can be generalised to a sequence of higher orientability conditions in different ways, of which we here discuss three. The first two are well-studied, whereas the third notion is new.
The first generalisation of orientability is given by the property of being $k$\emph{-parallelizable}. A manifold is $k$-parallelizable if the classifying map of the tangent bundle $\tau$ can be lifted to the $k$-connected cover of $BO(n)$. 
These connected covers form the Whitehead tower of $BO(n)$, of which each stage successively kills a higher homotopy group. By Bott periodicity, $BO = \lim_{n}BO(n)$ has a non-trivial homotopy group in dimensions $0,1,2,4 \,\,(mod \,8)$, and therefore non-trivial stages of the Whitehead tower occur only in these dimensions. 
The first few non-trivial stages have special names:
\begin{equation}
BO(n) \leftarrow BSO(n) \leftarrow BSpin(n) \leftarrow BString(n) \leftarrow B5brane(n) \leftarrow... \nonumber
\end{equation}
For example, $BString(n)$ is the stage at which the fourth homotopy group is killed, and it is 7-connected. If $\tau$ admits a lift to $BString(n)$, we call the manifold \emph{stringable}, which corresponds to 7-parallelizable. For a spinable manifold, this requires $\tau$ to be trivial on the fourth homotopy group, which is a strictly stronger condition than requiring that $w_4=0$. The obstruction class to finding a string structure on a spin manifold is $\frac{p_1}{2} \in H^4(M; \Z)$. 
If a manifold is $k$-parallizable for $k=n$, the dimension of the manifold, then the manifold is \emph{parallelizable} or \emph{frameable}.

A second notion of ``higher'' orientations arises from considering orientability with respect to different (co)homology theories (see for example \cite{Rudyak1998}, chapter V). 
There are occasionally links between the notion of being $k$-parallelizable and being orientable with respect to a certain cohomology theory. For example, a smooth manifold admits an orientation in real $K$-theory if and only if it is spinable (\cite{Atiyah1964}). If a manifold is stringable, then it admits an orientation for $tmf$, the universal  elliptic cohomology theory (\cite{Ando2010}).

We now define a third sequence of orientability conditions that can also be viewed as a natural generalization of orientability and spinability:

\begin{df}
	We call a manifold $k$-orientable if $w_i=0$ for $0<i<2^k$.
\end{df}

A manifold is 1-orientable if and only if it is orientable. A manifold is 2-orientable if and only if it is spinable. If a manifold is stringable, then it is 3-orientable, but the converse is not true. A counterexample is given by the manifold $\mathbb{CP}^3$, all of whose Stiefel-Whitney classes vanish although it is not stringable (\cite{Douglas2010}). Generally, 
increasing parallelizability by one \emph{non-trivial} step in the Whitehead tower implies increasing orientability by one, but the first poses a strictly stronger condition (for details see Lemma \ref{parallel}).

\vs

Orientability bears a nice relationship to a very classical invariant of topological spaces, the Euler characteristic $\chi$. It is a simple consequence of Poincar\'e duality that manifolds with an odd dimension have vanishing $\chi$. For orientable manifolds however, $\chi$ is furthermore even in dimensions that are not a multiple of 4. This pattern continues: spinable manifolds have an even $\chi$ in dimensions that are not a multiple of 8. This is implied by Ochanine's theorem on the divisibility of the signature of spin manifolds in dimension $8k+4$ (\cite{Ochanine1981}). The generalisation of these statements to higher $k$-orientability is the main result of this paper:

\begin{thm}[Corollary \ref{mainthm2}]\label{mainthm}
	A $k$-orientable manifold or Poincar\'e complex $M$ ($k \geq 0$) has an even Euler characteristic $\chi (M)$ if its dimension is not a multiple of $2^{k+1}$. 
\end{thm}

This theorem follows as a corollary from the following stronger statement about the vanishing of a large number of Wu classes.

\begin{thm}[Theorem \ref{morewus2}]\label{morewus}
	Let $M^{n}$ be a manifold or Poincar\'e complex of dimension $n$.  If $M$ is $k$-orientable for some $k$, then the Wu classes $v_l$ vanish for all $l$ such that $2^k \nmid l$.
\end{thm}

The proof is based on a relation within the Steenrod algebra derived from recursive application of the Adem relations.

In the 50's and 60's of the last century, the problem of relations between Stiefel-Whitney classes of manifolds was studied in \cite{Dold1956, Massey1960, Brown1964}. The question posed was to find all polynomials in Stiefel-Whitney classes of homogeneous cohomological degree that vanish for all manifolds of a given dimension $n$. Dold answered this question for degree $n$ polynomials, and Brown and Peterson gave a complete although indirect description of this set of polynomials for arbitrary degree in terms of a right action of the Steenrod algebra on the cohomology of $BO$. In the current work, we partially answer the subtly different question ``if $w_i=0$ for $0<i<2^k$, does $w_n$ vanish?'' Indeed, Brown and Peterson found an extended set of relations when considering the case of orientable manifolds. It would be interesting to study the more general question ``if $w_i=0$ for $0<i<2^k$, which $w_j$ can be non-vanishing?'', as well as to make the relations found in \cite{Brown1964} more explicit.

Theorem \ref{mainthm} in particular implies that 3-orientable and therefore stringable manifolds have an even Euler characteristic unless the dimension is a multiple of 16, that 4-orientable and therefore 5braneable manifolds have an even $\chi$ unless the dimension is a multiple of 32, et cetera.
One might wonder whether the theorem is strict: whether $k$-orientable manifolds with an odd Euler characteristic exist in all dimensions that are a multiple of $2^{k+1}$. 
There are non-orientable manifolds with odd $\chi$ in every even dimension as $\chi (\mathbb{RP}^{2m})=1$. Similarly, even-dimensional complex projective spaces $\mathbb{CP}^{2m}$, even-dimensional quaternion projective spaces $\mathbb{HP}^{2m}$, and powers of the octonion projective plane or Cayley plane $(\mathbb{OP}^{2})^m$ are examples of manifolds with an odd Euler characteristic that are orientable and of dimension $4m$, spinable of dimension $8m$, and stringable of dimension $16m$ respectively. Hence the theorem is strict for $k=0,1,2,3$. However, to the author's knowledge, it is an open question whether there exist 4-orientable manifolds with an odd Euler characteristic. In particular, one might wonder whether there exist 8-connected manifolds with an odd Euler characteristic.

\subsubsection*{Conventions}
Throughout we will be considering (co)homology with $\Z/2$ coefficients, hence from here on we write $H^n(X)$ for $H^n(X;\Z/2)$. Manifolds are considered to be closed and without boundary.

\subsubsection*{Acknowledgements}
I would like to thank Oscar Randal-Williams, Chris Douglas, Andr\'e Henriques, Diarmuid Crowley, Markus Land, Fabian Hebestreit and my supervisor Ulrike Tillmann for their comments and support.

\section{Preliminaries}

This section recalls definitions and some properties of Steenrod squares, Stiefel-Whitney classes and Wu classes.

\subsection{Steenrod Algebra}

Steenrod squares are cohomology operations on $\Z/2$ cohomology. 
We recall from \cite{Steenrod} their axiomatic definition:

\begin{enumerate}
	\item $Sq^i: H^n(X) \rightarrow H^{n+i}(X)$ are cohomology operations, i.e. they are natural with respect to maps of spaces.
	\item $Sq^0$ is the identity homomorphism.
	\item $Sq^{|x|}(x) = x \smile x$.	
	\item If $i > |x|$, where $|x|$ is the degree of $x$, then $Sq^i (x)=0$.
	\item The Steenrod squares obey the \emph{Cartan formula}:
	\begin{equation}
	Sq^k (x \smile y) = \sum_{i+j=k} Sq^i (x) \smile Sq^j (y) 
	\end{equation}
\end{enumerate}

It follows from these axioms that the Steenrod squares obey the \emph{Adem relations}: for all $a$ and $b$ such that $a < 2b$, we have
	\begin{equation}
	Sq^a  Sq^b =  \sum_{c=0}^{\lfloor a/2\rfloor} {b-c-1 \choose a-2c} Sq^{a+b-c}  Sq^c.
	\end{equation}

As a result of the Adem relations, certain composites of Steenrod squares always vanish, for example $Sq^{2n-1} Sq^n=0$, as for $a=2n-1$ and $b=n$ all binomials on the right hand side of the Adem relations vanish.
Another consequence of the Adem relations is that individual squares can be decomposed into a sum of composites of smaller squares, unless they are of the form $Sq^{2^i}$ for some $i$. 
Steenrod squares, subject to the Adem relations, generate a graded Hopf algebra over $\mathbb{F}_2$ called the \emph{Steenrod algebra} $\mathcal{A}$. 

One often groups the Steenrod operations into the \emph{total Steenrod square} acting on $\bigoplus_i H^i(X)$ given by 
\begin{equation}
Sq = Sq^0+Sq^1 +Sq^2 + ...
\end{equation}

\subsection{Stiefel-Whitney classes}\label{SWsection}

Stiefel-Whitney classes are $\Z/2$ characteristic classes of vector bundles. We recall some properties from \cite{May}, Chapter 23.
An $n$-dimensional vector bundle $\mathbb{R}^n \rightarrow \xi \rightarrow M$ over a manifold $M$ can be viewed as the pull-back of the universal bundle $\mathbb{R}^n \rightarrow \gamma_n \rightarrow BO(n)$ under a classifying map which we shall also denote $\xi$. 
\begin{center}
	\begin{tikzcd}
		\xi \arrow{r} \arrow{d}  & \gamma_n\arrow{d}\\
		M \arrow{r}{\xi} & BO(n)
	\end{tikzcd}
\end{center}

The $\Z/2$ cohomology of $BO(n)$ is given by:

\begin{equation}
H^* (BO(n)) = \Z/2 [w_1, w_2, ... , w_n],
\end{equation}
where $w_i$ with $|w_i|=i$ are the universal Stiefel-Whitney classes. We can ask what the image of this cohomology is under the classifying map $\xi$.

\begin{df}
	Given a vector bundle $\xi$ over a manifold $M$, the $i$-th Stiefel-Whitney class $w_i(\xi) \in H^{i}(M)$ is defined as 
	\begin{equation}
	w_i(\xi) = \xi^*(w_i), 
	\end{equation}
	for $w_i \in H^* (BO)$ the universal $i$-th Stiefel-Whitney class.
\end{df}

By the Stiefel-Whitney classes of a manifold we mean the Stiefel-Whitney classes of its tangent bundle. 
The Stiefel-Whitney classes of a manifold obey the Wu formula:

\begin{equation}\label{Wuformula}
Sq^i (w_j) = \sum_{t=0}^i {j+t-i-1 \choose t} w_{i-t} w_{j+t}.
\end{equation}

\subsection{Wu classes}

Let $M$ be an $n$-dimensional \emph{Poincar\'e complex} (see for example \cite{Ranicki2002}) with fundamental class $\mu \in H_n(M)$ such that 
\begin{equation}
-\cap \mu : H^*(M) \rightarrow H_{n-*}(M) \nonumber
\end{equation}
is an isomorphism. A topological manifold is a particular example of a Poincar\'e complex, as is any finite CW complex that is homotopy equivalent to a manifold.

Let
\begin{equation}
\langle .,. \rangle : H^i(M) \times H_i (M) \rightarrow \Z/2
\end{equation}
be the pairing of cohomology and homology. $M$ satisfies Poincar\'e duality for $\Z/2$ (co)homology, hence we have
\begin{equation}
Hom(H^{n-k}(M),\Z/2) \cong H_{n-k}(M) \cong H^k(M),
\end{equation}
where a cohomology class $y \in H^k(M)$ corresponds to the homomorphism $x \mapsto \langle y \smile x , [M] \rangle$. In particular, one element of $Hom(H^{n-k}(M),\Z/2)$ is given by $x \mapsto \langle Sq^k(x), [M] \rangle$. Under the above isomorphism, this corresponds to a degree $k$ cohomology class $v_k$ called the $k$-th Wu class, with
\begin{equation}\label{defWuclass}
\langle v_k \smile x, [M] \rangle =\langle Sq^k( x), [M] \rangle,
\end{equation}
for all $x \in H^{n-k}(M)$.
We can define the total Wu class as the formal sum 
\begin{equation}
v = 1 + v_1+v_2 +... + v_n
\end{equation}
This allows us to replace equation (\ref{defWuclass}) by 
\begin{equation}
\langle v \smile x, [M] \rangle =\langle Sq( x), [M] \rangle,
\end{equation}
where $Sq$ is the total Steenrod square and $x \in H^*(M)$.

Because $Sq^i (x)=0$ if $i > |x|$, an $n$-dimensional manifold only admits non-zero Steenrod operations for $i \leq \frac{n}{2}$. Hence a manifold has maximally $\frac{n}{2}$ non-zero Wu classes.

\vs

For a smooth manifold $M$ of dimension $n$, the Stiefel-Whitney classes of the manifold can be determined from the Steenrod squares and Wu classes. Define the total Stiefel-Whitney class as the formal sum
\begin{equation}
w= 1+ w_1 + w_2 + ... + w_n.
\end{equation}
The total Stiefel-Whitney class $w$ is then determined by the total Steenrod square $Sq$ and the total Wu class $v$ via the relation (see for example \cite{May}, Chapter 23):
\begin{equation}
w = Sq(v).\label{SW-Wu}
\end{equation}
This formula can be used to define Stiefel-Whitney classes for a Poincar\'e complex.

\section{Vanishing top Wu class implies even Euler characteristic}
As a key step in the proof of Theorem \ref{morewus} we will need the statement that the vanishing of the top Wu class of an even-dimensional manifold or Poincar\'e complex implies that the Euler characteristic is even. The vanishing of the top Wu class is in fact a stronger condition than having an even Euler characteristic. 
This is a classical result, but will be recalled in this section.

\begin{thm}[\cite{MilnorStasheff}]
For a smooth compact manifold $M^n$, the top Stiefel-Whitney class $w_n$ is the modulo 2 reduction of an integral characteristic class called the Euler class and it follows that the Stiefel-Whitney number $\langle w_n, [M] \rangle$ equals $\chi(M)$ modulo 2, hence $w_n=0$ if and only if $\chi(M)$ is even.
\end{thm}

For a manifold of even dimension $2n$, equation \ref{SW-Wu} implies that the top Stiefel-Witney class $w_{2n}= Sq^n v_n = v_n^2$, hence $v_n=0$ implies $w_{2n}=0$, which implies that the Euler characteristic is even.

In the more general case of a Poincar\'e complex, one can define Stiefel-Whitney classes via the Wu formula. One can then show that it still holds that the top Stiefel-Whitney class $w_{2n}$ is zero if and only if the Euler characteristic is even.
We here take a different route to illustrate an alternative proof method.
In order to avoid the use of the Euler class we will prove that the vanishing of $v_n$ implies an even Euler characteristic directly. For this we will first give a proof of the statement that a symplectic vector space over any field, in particular also $\mathbb{F}_2$, has even dimension. This statement is very classical (see for example \cite{Jacobson}, section 6.2). We call a vector space $V$ over a field $k$ \emph{symplectic} if it is endowed with a bilinear form $A:V^n \times V^n\rightarrow k$ that is non-degenerate, antisymmetric and has vanishing diagonal. The last condition is superfluous if $char(k)\neq 2$.

\begin{lem}\label{bilinear} 
	Let $V^n$ be an $n$-dimensional vector space over a field $k$ and $A:V^n \times V^n\rightarrow k$ an antisymmetric bilinear form with vanishing diagonal, i.e.
	\begin{equation}
	A(v,v)=0 \hspace{1cm} \forall v \in V^n.
	\end{equation}
	If $n$ is odd, then $A$ is degenerate i.e. $\det A =0$.
	
\end{lem}

\begin{proof}
	Consider the formula for the determinant:
	\begin{equation}
	\det(A)=\sum_{\sigma \in S_n} \text{sgn}(\sigma) \prod_{i=1}^{n} a_{i, \sigma(i)},
	\end{equation}
	where $S_n$ is the symmetric group on $n$ letters and $a_{i,j}$ are the matrix elements of $A$ in a given basis of $V^n$. As the diagonal of $A$ expressed as a matrix in this basis vanishes, any permutation $\sigma$ that leaves one or more indices unchanged does not contribute to the sum. Hence, we can restrict to the set of permutations $\Sigma_n$ containing no cycles of length 1. Within this subset, an element is of order two (its own inverse) if and only if it consists of disjoint cycles of length two. Since $n$ is odd, there us no such element. Hence $\Sigma_n$ can be partitioned as $\Sigma_n^+ \sqcup \Sigma_n^-$ such that $\Sigma_n^-$ consists of the inverses of the elements in $\Sigma_n^+$.
	The determinant formula then becomes:
	
	\begin{equation}
	\det(A)=\sum_{\sigma \in \Sigma_n^+} \text{sgn}(\sigma) \prod_{i=1}^{n} a_{i, \sigma(i)} +\text{sgn}(\sigma^{-1}) \prod_{i=1}^{n} a_{i, \sigma^{-1}(i)}. \label{products}
	\end{equation}
	
	Note that $\text{sgn}(\sigma)=\text{sgn}(\sigma^{-1})$. Furthermore, as $A$ is antisymmetric and $n$ odd, $\prod_{i=1}^{n} a_{i, \sigma^{-1}(i)} =  - \prod_{i=1}^{n} a_{\sigma^{-1}(i),i} = - \prod_{i=1}^{n} a_{i, \sigma(i)}$, by rearranging the terms in the product. Hence the two products in equation (\ref{products}) cancel and hence $det(A)=0$.
\end{proof}


\begin{thm}\label{topwu}
Let $M^{2n}$ be a Poincar\'e complex of dimension $2n$. Then $v_n=0$ implies that $\chi(M)$ is even.
\end{thm}

\begin{proof}
	Define the (anti-)symmetric bilinear form 
\begin{equation}
\omega: H^n(X)\times H^n(X)\rightarrow \mathbb{Z}/2 \nonumber
\end{equation}
given by $\omega(x,y) =\langle (x \smile y),\mu  \rangle$.  

As $x^2  =Sq^n(x) =v_n \smile x =0$ for all $x \in H^n(X)$, the diagonal of the $\omega$ is zero. It is furthermore non-degenerate because of Poincar\'e duality and because we are working over a field. 
Therefore $H^n(X)$ is a symplectic vector space, and thus has even dimension by Lemma \ref{bilinear}. The parity of the Euler characteristic of a Poincar\'e complex is determined by the parity of the dimension of the middle cohomology, hence it is even.
\end{proof}

\section{$k$-orientability}

A manifold is called $k$-orientable if $w_i=0$ for $0<i<2^k$.
This definition is motivated by \cite{Douglas2010}, where it is noted that the condition $w_i=0$ for $0<i<2^k$ for a manifold $M$ is equivalent to the condition that $H^*(M)$ is Poincar\'e dual with respect to the subalgebra $\mathcal{A}(k-1)$ of the Steenrod algebra spanned by $Sq^1, Sq^2,..., Sq^{2^{k-1}}$. 
The latter means that the actions of these squares are symmetrically distributed upon reflection in the middle dimension, where we should take every operation to its image under the canonical anti-homomorphism $\chi$ from $\mathcal{A}(k)$ to itself (see \cite{Douglas2010} for a detailed explanation).

The well-known fact that the Stiefel-Whitney classes of degree a power of two generate all Stiefel-Whitney classes of a manifold is a consequence of the Wu formula and Lemma \ref{claim}. The latter we will also need for the proof of the main theorem, hence we include both proofs.

\begin{lem} \label{power2}
	The Stiefel-Whitney classes of a smooth manifold are generated by those of the form $w_{2^i}$.
\end{lem}

Lemma \ref{power2} implies that a $k$-orientable smooth manifold is $k+1$-orientable precisely if the class $w_{2^{k}}$ vanishes. This motivates why in the definition of $k$-orientability, it is natural to let the number of vanishing Stiefel-Whitney classes double as we increase $k$ by one. It also follows that in order for a smooth manifold to be $k$-orientable it is enough to require $w_i$ to vanish for $i\leq 2^{k-1}$.

\begin{cor}
	A smooth manifold is $k$-orientable if and only if $w_i=0$ for $0<i\leq 2^{k-1}$.
\end{cor}

To prove Lemma \ref{power2}, we need the following combinatorial lemma.

\begin{lem}\label{claim}
	
	\begin{equation}{2^k m-1 \choose b} \equiv 1 \,\,\,(mod\, 2) \end{equation} 
	for all $0\leq b \leq 2^{k-1}$ and $m, k \geq 1$.
\end{lem}

\begin{proof}
	By Lucas's theorem (\cite{Lucas1878}), the value of a binomial ${a \choose b}$ modulo a prime $p$ can be calculated by considering the base $p$ expansions of $a$ and $b$:
	
	\begin{equation}
	{a \choose b} \equiv \prod_i {a_i \choose b_i}  \,\,\,(mod\, p), \label{Lucas}
	\end{equation}
	
	where $a_i$ and $b_i$ are the corresponding base $p$ digits of $a$ and $b$, and we use the convention that ${a_i \choose b_i}=0$ if $a_i<b_i$.
	
	Consider the case $p=2$ with $a=2^k m-1$ and $b \leq 2^{k-1}$. 
	Let $m = \sum_i m_i 2^i $ be the binary expansion of $m$, and $i_0$ the index of the lowest non-zero digit of $m$. Then 
	\begin{equation}
	a = 2^k m -1 = \sum_i m_i 2^{i+k} -1 =  \sum_{i>i_0} m_i 2^{i+k} + 2^{i_0 + k} -1,
	\end{equation}
	hence the binary expansion of $a$ ends in a sequence of 1's: $a_i=1$ for $0\leq i \leq i_0 + k-1$. On the other hand, $b_i=0$ for $i \geq k$. Hence every term in equation (\ref{Lucas}) is of the form ${1 \choose 0}$, ${1 \choose 1}$ or ${0 \choose 0}$, each of which equals 1 and hence ${a \choose b}\equiv 1 \,\,\,(mod\, 2)$.
\end{proof}

\begin{proof}[Proof of Lemma \ref{power2}.]
	We can rearrange the Wu formula (\ref{Wuformula}) as follows:
	\begin{equation}\label{Wuformula}
	{j-1 \choose i} w_{j+i} = Sq^i (w_j) + \sum_{t=0}^{i-1} {j+t-i-1 \choose t} w_{i-t} w_{j+t}.
	\end{equation}
	Hence we can rewrite $w_{n}$ in terms of Stiefel-Whitney classes of a lower degree precisely if we can split up $n$ as $i+j$ with $i,j\geq1$ in such a way that ${j-1 \choose i}$ is non-zero modulo 2.
	If $n$ is not a power of 2, then $n = 2^k + i$ for some $i<2^k$, setting $j=2^k$. Applying Lemma \ref{claim} for $m=1$ and $b=i$, we see that ${j-1 \choose i}\equiv 1 \,\,\,(mod\, 2)$. 
	
	If $n$ is a power of 2, then $i+j-1$ has only 1's in its binary expansion, hence every factor in the product given by equation (\ref{Lucas}) computing ${j-1 \choose i}$ modulo 2 is of the form ${1 \choose 0}\equiv 1$ or ${0 \choose 1} \equiv 0$. Then the only choice of $i$ for which ${j-1 \choose i}$ modulo 2 is non-zero is then $i=0$, in which case the Wu formula states $w_{n} = w_{n}$.
	
\end{proof}

The lemma below shows that we could also have used the Wu classes rather than the Stiefel-Whitney classes to define $k$-orientability.

\begin{lem}\label{orientableWu}
	The following are equivalent:
	\begin{enumerate}[(i)]
		\item A manifold is k-orientable
		\item $v_i=0$ for $0<i \leq 2^{k-1}$
		\item $v_i=0$ for $0<i < 2^k$
	\end{enumerate}

\end{lem}
\begin{proof}
	It follows from equation (\ref{SW-Wu}) that
	\begin{equation}
	w_i = \sum_{j=0}^{\lfloor i/2 \rfloor} Sq^j (v_{i-j}).
	\end{equation}

	From this formula it follows immediately that $v_i=0$ for $i\leq 2^{k-1}$ implies $w_i=0$ for $i \leq 2^{k-1}$, and that $v_i=0$ for $i< 2^k$ implies $w_i=0$ for $i < 2^k$, giving the implications $(ii) \Rightarrow (i)$ and $(iii) \Rightarrow (i)$.
	
	We will show by induction that the lowest non-vanishing Stiefel-Whitney class equals the lowest non-vanishing Wu class, i.e. if $w_i = 0$ for $0<i<n$, then $w_i=v_i$ for $i\leq n$. Note that $w_1=v_1$. Assume that $w_i = 0$ for $0<i<n$ and by induction hypothesis $v_i=w_i=0$ for $0<i\leq n-1$.
	Then $$w_n = v_n + \sum_{j=1}^{\lfloor i/2 \rfloor} Sq^j (v_{i-j}) = v_n.$$
It follows that $w_i=0$ for $i<n$ implies $v_i=0$ for $i<n$, hence $(i) \Rightarrow (ii)$ and $(i) \Rightarrow (iii)$.
	
\end{proof}

When comparing $k$-orientability to $k$-parallelizability, one would naively expect that only $2^{k-1}$-parallelizability would imply $k$-orientability of a manifold. However, by the classical work of Stong (\cite{Stong1963}), the Stiefel-Whitney classes of the $n$-connective covers of $BO$ vanish exponentially with $n$, with the number of vanishing $w_i$ doubling at each non-trivial stage (occurring in degrees 0, 1, 2 and 4 modulo 8). Indeed, we have the following theorem:

\begin{thm}[\cite{Stong1963}]\label{parallel}
	Let $\phi(n)$ be the number of integers $s$ with $1\leq s\leq n$ such that $s \equiv 0,1,2,4 \,\,\,(mod\, 8)$. If a manifold $M$ is $n$-parallelizable for $n$ with $\phi(n)\geq k$, then $M$ is $k$-orientable. 
\end{thm}

In other words, a lift of the classifying map of the tangent space of $M$ to the $k^{th}$ \emph{non-trivial} stage in the Whitehead tower of $BO(n)$ precisely implies $k$-orientability.

\vs

In a similar way as one could consider lifts to a sequence of connected covers of $BO(n)$ to give a homotopy-theoretic definition of $k$-parallisability, one can build a sequence of covers of $BO(n)$ such that lifts give rise to a $k$-orientation:
\begin{equation}
BO(n) \leftarrow BSO(n) \leftarrow BSpin(n) \leftarrow BOr_3(n) \leftarrow BOr_4(n) \leftarrow... \nonumber
\end{equation}

Here we formally define $BOr_k(n)$ up to homotopy by taking the homotopy fibre of the maps that classify the appropriate Stiefel-Whitney classes:

\begin{equation}\label{fibreseq}
BOr_k(n) \rightarrow BOr_{k-1}(n)  \xrightarrow{w_{2^{k-1}}} K(\mathbb{Z}/2, 2^{k-1}).
\end{equation}

The class $w_{2^{k-1}}$ is non-trivial in the cohomology of $BOr_{k-1}= \lim_n BOr_{k-1}(n)$ by the lemma below.
\begin{lem}
$w_{2^{k-1}}$ is a non-zero characteristic class of the map $BOr_{k-1} \rightarrow BO$.
\end{lem}
\begin{proof}
For $n$ such that $\phi(n) = k-1$ and $BO[n]$ the $n$-connected cover of $BO$, consider the diagram

\begin{center}
	\begin{tikzcd}
		\text{ }  & BOr_{k-1} \arrow{d} & \\
		BO[n]\arrow{r}\arrow[dotted]{ru} & BO  \arrow{r}{w_{2^{k-1}}} & K(\mathbb{Z}/2, 2^{k-1})
	\end{tikzcd}
\end{center}

By the work of Stong in \cite{Stong1963}, $BO[n]$ has the propery that its Stiefel-Whitney classes $w_i$ vanish for $0<i<2^{k-1}$, but $w_{2^{k-1}}$ does not vanish. Hence the map $BO[n] \rightarrow BO$ admits a lift to $BOr_{k-1}$, while the composite of the maps along the bottom of the diagram is not null-homotopic. Therefore the composite $BOr_{k-1}\rightarrow BO \xrightarrow{w_{2^{k-1}}} K(\mathbb{Z}/2, 2^{k-1})$ is not null-homotopic either.	
\end{proof}

By Lemma \ref{power2} it suffices to kill only the Stiefel-Whitney classes of the form $w_{2^{i}}$ for $i<k$ to classify $k$-orientable smooth manifolds.

For $k\geq3$, the long exact sequence in homotopy corresponding to (\ref{fibreseq}) reduces to the short exact sequence:

\begin{equation}
0 \rightarrow \pi_{2^{k-1}} (BOr_k) \rightarrow \Z \xrightarrow{w_{2^{k-1}}} \Z/2 \rightarrow 0,
\end{equation}

hence $\pi_{2^{k-1}} (BOr_k)= \Z$.

\section{Proof of Theorem \ref{mainthm}}

We will now prove our main result that $k$-orientable manifolds have an even Euler characteristic as long as their dimension is not a multiple of $2^{k+1}$. In order to do so, we first prove a technical result that describes a relation in the Steenrod Algebra, arising from repeated application of the Adem relations. It tells us how Steenrod squares decompose that are $k-1$ times even but not $k$ times even.

\begin{prop}\label{Lemma}
	Let $n=2^k m + 2^{k-1}$ for some $m\geq 0$, $k\geq 1$. Then
	
	\begin{equation}
	Sq^n =  \sum_{i=1}^{2^{k-1}} Sq^i  \alpha^{n-i},
	\end{equation}
	where $\alpha^{n-i}$ is a sum of composites of squares.
\end{prop}

\begin{proof}
	For $m=0$, $n = 2^{k-1}$, thus the result follows. 
	
	Assume $m\geq 1$.
	For $a$ and $b$ with $a<2b$, the Adem relations can be rearranged as follows:
	\begin{equation}\label{15}
	{b-1\choose a} Sq^{a+b} =Sq^a  Sq^b + \sum_{c=1}^{\lfloor a/2\rfloor} {b-c-1 \choose a-2c} Sq^{a+b-c}  Sq^c .
	\end{equation}
	
	Applying this to $a=2^{k-1}$, $b= 2^k m$, with $n = a+b$, gives:
	
	\begin{equation}
	{2^k m-1 \choose 2^{k-1}} Sq^{n} = Sq^{2^{k-1}}  Sq^{2^k m} + \sum_{c=1}^{2^{k-2}} {2^k m-c-1 \choose 2^{k-1}-2c} Sq^{n-c}  Sq^c  .\label{Sqn}
	\end{equation}
	
	By Lemma \ref{claim}, the binomial ${2^k m-1 \choose 2^{k-1}}$ is non-zero. Hence we have now expressed $Sq^n$ as the sum of a term with a small square in front, $Sq^{2^{k-1}}  Sq^{2^k m}$, and a remainder. Each term in the remainder contains $Sq^{n-c}$ for some $c$, with $1 \leq c \leq 2^{k-2}$. Define $i= 2^{k-2}-c$ such that $Sq^{n-c}  = Sq^{2^k m +i}$. We use strong induction on $i$ to show that each of these squares can be written in the desired form, as a sum of composites in which each term has a square on the left of degree less than or equal to $ 2^{k-1}$.
	
	The base case is given by $i=1$. Applying equation (\ref{15}) for $a=1$, $b= 2^k m$ gives:
	\begin{equation}
	{2^k m-1 \choose 1} Sq^{2^k m   + 1} = Sq^{2^k m   + 1} =Sq^{1}  Sq^{2^k m},
	\end{equation}
	which is of the desired form.
	
	Suppose that $Sq^{2^k m +i}$ can be written in the desired form for all $i<i_0$. We will show that $Sq^{2^k m +i_0}$ can also be written in the desired form. Applying (\ref{15}) for $a=i_0$, $b= 2^k m$ gives:
	\begin{multline}
	{2^k m-1 \choose i_0} Sq^{2^k m   + i_0} = Sq^{i_0}  Sq^{2^k m}\\
	+
	\sum_{c=1}^{\lfloor i_0/2\rfloor} {2^k m-c-1 \choose i_0-2c} Sq^{2^k m + i_0-c}  Sq^c.
	\label{reapply}
	\end{multline}
	By Lemma \ref{claim}, the binomial ${2^k m-1 \choose i_0}$ is non-zero. 
	The first term on the right hand side is of the right form, and by the hypothesis, each $Sq^{2^k m + i_0-c}$ can be written in the desired form. This concludes the induction.
	
	Hence we observe that all terms in equation (\ref{Sqn}) are of the form $Sq^i  \alpha'^{n-i}$ for some composite of squares $\alpha'^{n-i}$. Collecting the terms with the same $Sq^i$ in front and making use of the linearity of the squares renders the desired formula.
\end{proof}

\begin{thm}\label{morewus2}
	Let $M^{n}$ be a manifold or Poincar\'e complex of dimension $n$.  If $M$ is $k$-orientable for some $k$, then the Wu classes $v_l$ vanish for all $l$ such that $2^k \nmid l$.
\end{thm}

\begin{proof}
	By Lemma \ref{orientableWu}, $v_i=0$ for $0<i< 2^{k}$ as a consequence of the $k$-orientability of $M$. Hence, by the definition of $v_i$, the operations
	\begin{equation}
	Sq^i: H^{2n-i}(M) \rightarrow H^{2n}(M) \nonumber
	\end{equation}
	vanish for $i< 2^{k}$.

	Consider $v_l$ such that $2^k \nmid l$. Let $k'-1< k$ be the highest power of 2 dividing l, such that $l = 2^{k'} m'+ 2^{k'-1}$ for some $m'$. Then by Proposition \ref{Lemma}, 
	
	\begin{equation}
	Sq^l =  \sum_{i=1}^{2^{k'-1}} Sq^i  \alpha^{l-i},
	\nonumber
	\end{equation}
	
	for some operations $\alpha^{l-i}$. 	
We consider this square with target the top dimension:
	\begin{equation}
	Sq^l: H^{n-l} (M;\mathbb{Z}/2)\rightarrow H^{n} (M;\mathbb{Z}/2).
	\nonumber
	\end{equation}
In the decomposition of $Sq^l$, the operations 		
\begin{equation}
Sq^i: H^{n-i} (M;\mathbb{Z}/2)\rightarrow H^{n} (M;\mathbb{Z}/2)
\nonumber
\end{equation}
are of degree $0<i< 2^{k'-1}< 2^k$ with target the top dimension and hence vanish. Therefore $Sq^l$ with target the top dimension vanishes, which implies that $v_l=0$.
\end{proof}

Theorem \ref{mainthm} for $k\geq 1$ now follows as a corollary of the theorem above. Note that the case $k=0$ is a trivial consequence of Poincar\'e duality: odd-dimensional manifolds have zero, thus in particular even, Euler characteristic.

\begin{cor}\label{mainthm2}
	A $k$-orientable manifold or Poincar\'e complex $M^{2n}$ ($k \geq 1$) of dimension $2n$ has an even Euler characteristic $\chi (M)$ if $2^{k+1}\nmid 2n$. 
\end{cor}

\begin{proof}
We have that $2^k \nmid n$, hence the top Wu class $v_n$ vanishes by Theorem \ref{morewus}. By Theorem \ref{topwu}, this implies that the Euler characteristic is even.

\end{proof}

\bibliographystyle{alpha}
\bibliography{library}

\end{document}